\documentclass{article}

\usepackage[12pt]{extsizes}
\usepackage{amsfonts}
\usepackage{euscript}
\usepackage{amsfonts,amssymb,amsmath,amsbsy,amsthm,amscd}
\usepackage[T2A]{fontenc}
\usepackage[cp1251]{inputenc}
\usepackage[english]{babel} 

\theoremstyle{plain}
\newtheorem{theorem}{Proposition}
\newtheorem{corollary}{Corollary}
\newtheorem{thh}{Theorem}
\theoremstyle{definition}
\newtheorem{remark}{Remark}

\DeclareMathOperator{\rk}{rk}

\title{ Decomposing a Matrix into two Submatrices with Extremely Small Operator Norm}
\author{ I.V.Limonova\footnote{
e-mail: limonova\_irina@rambler.ru.
 }}
\date{Lomonosov Moscow State University}
 
\begin{document}
\maketitle
\begin{abstract}
We give sufficient conditions on a matrix A ensuring the existence of a partition of this matrix into two submatrices with extremely small norm of the image of any vector. Under some weak conditions on a matrix A we obtain a partition of A with the extremely small $(1,q)$--norm of submatrices.
\end{abstract}

\\

\normalsize
This paper is devoted to the estimates of operator norms of submatrices. The subject is actively being developed and finds various applications. The present work can be viewed as a continuation of \cite{1} discussing the $(2,1)$--norm case. This case was studied earlier for matrices with orthonormal columns in \cite{2}, where an analogue of the partition (\ref{partition}) (see below) with the extremely small $(2,1)$--norm of the corresponding submatrices was obtained.  Using the modified Lunin's method we prove an essential reinforcement of Assertion $4$ from \cite{1} and the generalization of Assertion $3$ to the case of the $(X, q)$--norm with $1\leq q<\infty$. We study the case of the $(1,q)$--norm in greater detail.

For an $N\times n$ matrix $A$, viewed as an operator from $l_p^n$ to $l_q^N$, we define the $(p, q)$--norm:
\begin{equation*}
  \left\| A\right\|_{(p,q)}=\sup\limits_{\left\| x \right\|_{l_p^n}\leq 1}\left\| Ax\right\| _{l_q^N}, \ 1\leq p, q\leq \infty.
\end{equation*}
In fact, Proposition \ref{assertion_norm} is proved here for a more general $(X,q)$--norm where $X$ is an $n$--dimensional norm space.

We use the following notation: $\rk(A)$ is the rank of a matrix $A$, $\left< N\right>$ is the set of natural numbers $1, 2,\ldots, N$; $v_i$, $i\in \left< N\right>$ stands for the rows of $A$,  $w_j$, $j\in\left<n\right>$ --- its columns. For a subset $\omega\subset\left<N\right>$  $A(\omega)$ denotes the submatrix of a matrix $A$ formed by the rows $v_i, i\in\omega$, $\overline{\omega}=\left<N\right>\setminus\omega$ .
 $( \cdot, \cdot)$ stands for the inner product in $\mathbb R ^n$,  $\left\| x\right\|_p$ is the norm of $x\in\mathbb{R}^n$ in $l_p^n$, $1\leq p\leq\infty$. For a norm space $X$ $\left\|\cdot\right\|_X$ is the norm on X.

The following condition is the counterpart of the condition on a matrix from \cite{1} in the case of an arbitrary $1\leq q<\infty$: 
\begin{equation}\label{cond}
 \forall x\in\mathbb R^n\ \  \forall i_0\in\left<N\right> \ \ |(v_{i_0}, x)|\leq\varepsilon\left(\sum_{i=1}^N|(v_i, x)|^q\right)^{1/q}.
\end{equation}

\begin{theorem}\label{assertion_pointwise}
Assume that an $N\times n$ matrix $A$ stisfies (\ref{cond})  with $0<\varepsilon\leq (\rk(A))^{-1/q}$ and $1\leq q<\infty$. Then there exists a partition 
\begin{equation}\label{partition}
\left<N\right>=\Omega_1\cup\Omega_2,\  \Omega_1\cap\Omega_2=\emptyset, 
\end{equation}
such that 
\begin{equation}\label{obtained_estimate}
  \left\| A(\Omega_k)x\right\|_q\leq \gamma\left\| Ax\right\|_q, \ \ \gamma=\frac{1}{2^{1/q}}+
  \frac{2+3\cdot 2^{-1/q}}{q}\left(\rk(A)\varepsilon^q\ln \frac{6q}{(\rk(A)\varepsilon^q)^{1/3}}\right)^{1/3}
  \end{equation}
  for any $x\in\mathbb{R}^n$ and $k=1,2$.
\end{theorem}

\begin{remark}
No one knows whether such a partition exists or not if $1<\rk(A)\varepsilon^q$. 
\end{remark}

\begin{proof}[Sketch of proof]
First, we prove Proposition \ref{assertion_pointwise} for the case of $\rk(A)=n$.
Denote $$\delta=\frac{(n\varepsilon^q)^{1/3}}{q}.$$
Let $X$ be the space $\mathbb{R}^n$ with the norm $\left\|x\right\|_X=\left\|Ax\right\|_q$ (it is a norm on $\mathbb{R}^n$ because $\rk(A)=n$). Let $S_X=\{x\in \mathbb{R}^n: \left\|x\right\|_X=1 \}$ be the unit sphere of $X$.
Let $\mathbb{Y}$ be a $\delta$--net in the norm $\left\|\cdot\right\|_X$ on the sphere $S_X$ with at most $\leq(3/\delta)^n$ elements. Suppose that it is not the case, then for every partition (\ref{partition})  there exists a vector
$x_1\in S_X$ such that
$$
\left\| A(\Omega_1)x_1\right\|_q>\gamma\left\| Ax_1\right\|_q
$$
(in this case let $\omega'=\Omega_1$, $x_{\omega'}=x_1$ ), or there exists a vector $x_2\in S_X$ such that
$$
\left\| A(\Omega_2)x_2\right\|_q>\gamma\left\| Ax_2\right\|_q
$$
(then we define $\omega'=\Omega_2$, $x_{\omega'}=x_2$ ). 
For every pair $(\Omega_1, \Omega_2)$ we find $\omega'$ and $x_{\omega'}$.
 Let $y_{\omega'}$ be one of the nearest to $x_{\omega'}$ vectors from the net $\mathbb{Y}$.
There are $2^{N-1}-1$ different partitions of the set $\left<N\right>$ into two nonempty parts. Therefore there exists a vector $y_0\in \mathbb{Y}$ such that the set $K=\{\omega' : y_0=y_{\omega'}\}$ is large enough:
\begin{equation}\label{K_est}
  |K|\geq(2^{N-1}-1)\left(\frac{\delta}{3}\right)^n \geq 2^N\left(\frac{\delta}{6}\right)^n. 
\end{equation}
(Here we assume that $n>1$, otherwise Proposition \ref{assertion_pointwise} is obvious.)
Therefore there is a vector $y_0\in S_X$ and at least 
$2^N(\delta/6)^n$ subsets $\omega'\subset\langle N\rangle$ for which $\left\|A(\omega')x_{\omega'}\right\|_q>\gamma \left\|Ax_{\omega'}\right\|_q$ and $\left\|y_0-x_{\omega'}\right\|_X<\delta$. 
Note that for $x\in S_X$ and $\omega\subset\langle N\rangle$ $\left\|A(\omega)x\right\|_q\leq \left\|A(\omega)\right\|_{(X,q)}\leq \left\|A\right\|_{(X,q)}$.

Below we assume that $\gamma<1$, otherwise (\ref{obtained_estimate}) is obviously true. As $\gamma<1$, for $\omega'\in K$ we obtain:
\begin{gather*}
\left\| A(\omega')y_0\right\| _q\geq \left\| A(\omega')x_{\omega'}\right\|  _q-\left\| A(\omega')(x_{\omega'}-y_0)\right\| _q>\\
>\gamma\left\| Ax_{\omega'}\right\| _q-\delta
\left\| A(\omega')\left\{\frac{x_{\omega'}-y_0}{\left\|x_{\omega'}-y_0\right\|_X}\right\}\right\|_q\geq
(\gamma\left\| Ay_0\right\| _q-\gamma\left\| A(x_{\omega'}-y_0)\right\|  _q)-
\end{gather*}
\begin{equation*}\label{24}
-\delta \left\| A\left(\frac{x_{\omega'}-y_0}{\left\|x_{\omega'}-y_0\right\|_X}\right)\right\| _q\geq \gamma\left\| Ay_0\right\| _q-2\delta \geq \gamma\left\| Ay_0\right\|_q-2\delta \left\| Ay_0\right\|_q=
\left\| Ay_0\right\|_q\left(\gamma-2\delta\right).
\end{equation*}

Since $y_0\in S_X$, we have used $\left\|Ay_0\right\|_q=\left\|y_0\right\|_X=1$ in the last inequality.
Let $R$ be an amount of subsets $\omega\subset\langle N\rangle$ for which
$$
\left\| A(\omega)y_0\right\|_q\geq \left(\gamma-2\delta\right)\left\| Ay_0\right\|_q
$$
holds. Let $K_1$ be the set of such subsets. Let us show that  $R<2^N(\delta/6)^n$, then we will come to the contradiction, and it will complete the proof of Proposition \ref{assertion_pointwise} in the case of  $\rk(A)=n$. Denote $M=3\cdot 2^{-1/q}$.
Since $\delta\leq\phi(n, \varepsilon)$, then for $\omega'\in K_1$ we have:
$$
\sum\limits_{i\in \omega'}{|(v_i,y_0)|^q}>(\gamma-2\delta)^q S> \left(\frac{1}{2^{1/q}}+M\phi(n, \varepsilon)\right)^q S\geq 
$$
$$
\geq
\left(\frac{1}{2}+q\frac{1}{2^{(q-1)/q}}M\phi(n, \varepsilon)\right)S=\left(\frac{1}{2}+q\frac{2^{1/q}}{2}M\phi(n, \varepsilon)\right)S.
$$
 $R$ can be estimated as in the proof of Assertion $3$ from \cite{1}.
 
Now, let a matrix have the rank $r<n$. Without loss of generality, we can assume that the vectors $w_1, \dots, w_{r}$ are linearly independent. 
It is clear that (\ref{cond}) holds for the matrix $\tilde{A}$, which consists from the first $r$ columns of $A$. We have $\rk{\tilde{A}}=r$, therefore there exists a partition of the form (\ref{partition}) such that (\ref{obtained_estimate}) holds. Let $w_j=\sum\limits_{i=1}^r \lambda_j^i w_i$. For a vector $x\in\mathbb{R}^n$ we construct the vector $\tilde{x}\in\mathbb{R}^r$ having coordinates $\tilde{x}_i=x_i+\sum\limits_{j=r+1}^n \lambda_j^ix_j$, then $Ax=\tilde{A}\tilde{x}$ and for $k=1,2$  $A(\Omega_k)x=\tilde{A}(\Omega_k)\tilde{x}$, so for the partition we have found (\ref{obtained_estimate}) also holds for the matrix $A$.
\end{proof}
\medskip

\begin{corollary}\label{assertion_pointwise}
Assume that an $N\times n$ matrix $A$ satisfies (\ref{cond}) with $0<\varepsilon\leq (\rk(A))^{-1/q}$ and $1\leq q<\infty$. Then there exists a partition (\ref{partition})
such that for any $x\in\mathbb{R}^n$ and $k=1,2$ we have
\begin{equation*}
    \left(\frac{1}{2}-\psi\right)\sum_{i\in\Omega_k}|(v_i, x)|^q\leq \sum_{i=1}^N|(v_i, x)|^q\leq \left(\frac{1}{2}+\psi\right)\sum_{i\in\Omega_k}|(v_i, x)|^q, 
    \end{equation*}
    where 
    $$
\psi=2^{q+1}
\left(\rk(A)\varepsilon^q\ln \frac{6q}{(\rk(A)\varepsilon^q)^{1/3}}\right)^{1/3}.
    $$
\end{corollary}

The following proposition is a simple corollary of Proposition \ref{assertion_pointwise}.

\begin{theorem}\label{assertion_norm} Let for an $N\times n$ matrix $A$ (\ref{cond}) hold for some $0<\varepsilon\leq (\rk(A))^{-1/q}$ and $1\leq q<\infty$.
Then there exists a partition (\ref{partition}) such that for  $k=1,2$ the following inequality holds
$$
\left\| A(\Omega_k)\right\|_{(X,q)}\leq \gamma \left\| A\right\|_{(X,q)},
$$
where $\gamma$ is defined in the formulation of Proposition \ref{assertion_pointwise}.
\end{theorem}

The following proposition is analogous to Proposition \ref{assertion_norm} for the $(1,q)$--norm. Let $e_j$, $j\in\langle n\rangle$ be the standard basis in $\mathbb{R}^n$.

\begin{theorem}\label{assertion_1_q_norm}
  If for an $N\times n$ matrix $A$ the
  inequality
  \begin{equation}\label{a_i_j_cond}
  |a_j^i|\leq \varepsilon \left\| w_j \right\|_q 
  \end{equation}
  holds for some  $1\leq q<\infty$ and $0<\varepsilon< 1$ and for every $i\in\langle N\rangle$, and $j\in\langle n\rangle$, then there exists a partition (\ref{partition}) such that for $k=1,2$ the following holds:
  
  $
  \text{a) }\left\| A(\Omega_k)\right\|_{(1,q)}\leq
  \left(\frac{1}{2}+\frac{3}{2}\varepsilon^{q/3}\ln^{1/3}{(4n)}\right)^{1/q}
  \left\| A\right\|_{(1,q)},$
  
  $
  \text{b) }\left\| A(\Omega_k)\right\|_{(1,q)}\leq
  \left(\frac{1}{2}+\frac{1}{2}\varepsilon^{q}\sqrt{N}(1+\log(\frac{n}{N}+1)^{1/2}\right)^{1/q}
  \left\| A\right\|_{(1,q)}, $
  
  $
  \text{c) } \left\| A(\Omega_k)\right\|_{(1,q)}\leq\left(\frac{1+n\varepsilon^q}{2}\right)^{1/q}\left\| A\right\|_{(1,q)}.
  $
  \end{theorem}
  \begin{remark} In Proposition \ref{assertion_1_q_norm} we need sufficiently weak conditions (compared to Proposition \ref{assertion_pointwise}) on the elements of a matrix.
  \end{remark}
  
  \begin{proof} Since 
  the function $\left\| Ax \right\|_q$ is convex, then the $(1, q)$--norm of a matrix is attained on one of the vectors from the standard basis. 
  
  The proof of a) is indeed close to the previous arguments from Proposition \ref{assertion_pointwise}, so here we only show the sketch. Assume that our proposition is not true and for each partition (\ref{partition}) there exists a number $k$ such that  $\left\| A(\Omega_k)\right\|_{(1,q)}>\left(1/2+(3/2)\varepsilon^{q/3}\ln^{1/3}{(4n)}\right)^{1/q}\left\| A\right\|_{(1,q)}$. Denote $\omega'=\Omega_k$. The  $(1,q)$--norm of the matrix $A_{\omega'}$ is attained on some vector  $e_{j_{\omega'}}$, $j_{\omega'}\in\langle n\rangle$, therefore the following holds:
   $ \sum\limits_{i\in \omega'}|a^i_{j_{\omega'}}|^q>\left(1/2+(3/2)\varepsilon^{q/3}\ln^{1/3}{(4n)}\right)\left\|w_{j_{\omega'}}\right\|^q$. 
   Like in the proof of Proposition \ref{assertion_pointwise}, there exists  $j_0\in \langle n\rangle$ such that the set $K=\{\omega' : j_{\omega'}=j_0\}$ is large enough:
  \begin{equation*}
  |K|\geq(2^{N-1}-1)/n> 2^{N-2}/n. \eqno(6)
  \end{equation*}
  It is easy to see that for every $\omega\in K$
  \begin{equation*}
   \sum\limits_{i\in \omega}|a^i_{j_0}|^q>\left(1/2+(3/2)\varepsilon^{q/3}\ln^{1/3}{(4n)}\right)\left\|w_{j_0}\right\|^q. \eqno(7) 
  \end{equation*}
  
  So, for the proof of a) it is enough to check that a number 
   $R$ of subsets
  $\omega\subset\langle N\rangle$ for which (7) holds is less than the right part of (6).
  The value $R$ is estimated as in the proof of Assertion $3$ from \cite{1}.
  \medskip
  
  To prove b) we use Corollary $5$ from \cite{3}.

  Let $\tilde w_j=(|a_j^1|^q,\dots, |a_j^N|^q)$ be a vector which is obtained from the $j$--th column of $A$ by raising the moduli of its coordinates to the power $q$. 
  For all $j\in\langle n\rangle$ $\left\|w_j\right\|_q^q\leq \left\|A\right\|_{(1,q)}$, so (\ref{a_i_j_cond}) implies that $\left\|\tilde w_j\right\|_{\infty}\leq \varepsilon^q\left\|A\right\|_{(1,q)}^q$.  Then due to Corollary from \cite{3} mentioned above there exists such a vector $\xi=(\xi_1,\dots,\xi_N)\in\mathbb R^N$, whose coordinates have modulus $1$ such that for every $j\in\langle n\rangle$ the following inequality holds:
  \begin{equation*}
      \left|( \tilde w_j, \xi )\right|\leq \varepsilon^q\sqrt{N}\left(1+\log\bigl(\frac{n}{N}+1\bigr)\right)^{1/2}\left\|A\right\|_{(1,q)}^q. 
  \end{equation*}
  Let $\Omega_1=\{i\in\langle N\rangle : \xi_i=1\}$, $\Omega_2=\langle N\rangle\backslash \Omega_1=\{i\in\langle N\rangle : \xi_i=-1\}$. Let us check b). Denote $\theta=\sqrt{N}\left(1+\log\bigl(\frac{n}{N}+1\bigr)\right)^{1/2}$.  For $k=1,2$ there exists $j_0^k\in\langle n\rangle$ such that
  \begin{gather*}
  \left\|A(\Omega_k)\right\|_{(1,q)}^q=
  \sum\limits_{i\in\Omega_k}|a_{j_0^k}^i|^q
  \leq
  \frac{1}{2}\left(\sum\limits_{i\in\langle N\rangle}|a_{j_0}^i|^q+\varepsilon^q\theta\left\|A\right\|_{(1,q)}^q\right)
  \leq \left(\frac{1}{2}+\frac{1}{2}\varepsilon^{q}\theta\right)\left\|A\right\|_{(1,q)}^q,
  \end{gather*}
  as required.
  
  To prove c) we apply the following theorem.
  
  \begin{thh}[\cite{4}, p. 287] Let $A_1,\dots, A_n$ be sets in $\mathbb{R}^n$ with finite Lebesgue measure, then there exists a ~hyperplane  $\pi$ which divides the measure of each of them in half. 
  \end{thh}

  Let $M=\max\limits_{i,j}\lbrace|a_j^i|^q\rbrace+1$.
  One can put in $\mathbb{R}^n$ $N$ cubes with sides equal to $M$ and parallel to the axes such that every hyperplane intersects at most $n$ of them. (It follows from the existence of $N$ points of the general position in $\mathbb{R}^n$ and the continuity of the equation of a plane.) Let us numerate these cubes. For  $i\in\langle N\rangle$ let $u_i$ be the vertex of  $i$--th cube with the smallest coordinates.
 For each entry $a_j^i$ of the matrix we define a parallelepiped $\widetilde{P_j^i}=  [0,1]^{j-1}\times [1, 1+|a_j^i|^q]\times [0,1]^{n- j}$. We put $n$ rectangular parallelepipeds defined by the entries of the row $v_i$ ($P_j^i=u_i+ \widetilde{P_j^i}$) into the cube with number $i$. Note that $\mu(P_j^i)=|a_j^i|^q$. We call the set of  $P_j^i$, $j\in\langle n \rangle$ for a fixed $i$ by an $i$--th $``$angle$"$. 
  
For $j\in\langle n \rangle$ let 
  $A_j = \bigcup\limits_{i\in\langle N\rangle}{P_j^i}$.
Applying the theorem mentioned above to $A_j$, we get a hyperplane  $\pi$ which divides in half the measure of each $A_j$. Let $P_1$ and $P_2$ be halfspaces into which $\pi$  divides $\mathbb{R}^n$. By construction $\pi$ intersects at most $n$ cubes, consequently at most $n$ $``$angles$"$. It is clear now how to obtain a partition (\ref{partition}). We put the indices of the $``$angles$"$ which entirely belong to  $P_1$ (or $P_2$) in $\Omega_1$ (in $\Omega_2$ correspondingly). We put the indices of the $``$angles$"$ which intersect both $P_1$ and $P_2$ in $\Omega_1$. Let $G$ be the set of such indices.
Let us show that for every  $j\in \langle n \rangle$ the $l_q^N$--norm of the column $w_j$ will decrease at least $\left(\frac{1+n\varepsilon^q}{2}\right)^{1/q}$ times under the partition. It will prove our proposition.
Since $\pi$ divides in half the measure of $A_j$, then 
  $$\sum\limits_{i\in \Omega_1 \backslash G } |a_j^i|^q + V_1 = \sum\limits_{i\in \Omega_2} |a_j^i|^q + V_2,$$
 where  $V_k$, $k=1,2$ stands for the volume of $\cup_{i\in\langle G\rangle}P_j^i\cap P_k$. 
 From (\ref{a_i_j_cond}) and due to the fact that $\pi$ intersects at most $n$ $``$angles$"$, we have the following inequality:
  $$V_1+V_2 \leq n\varepsilon^q \sum\limits_{i\in \langle N\rangle}{|a_j^i|^q},$$ 
  so for $k=1,2$
  $$
  \sum\limits_{i\in \Omega_k} |a_j^i|^q \leq \frac{1+n\varepsilon^q}{2}\sum\limits_{i\in \langle N\rangle} |a_j^i|^q.
  $$
  Thus, Proposition \ref{assertion_1_q_norm} is proved.
  \end{proof}
  
 The following proposition shows that there is a case when (\ref{a_i_j_cond}) holds for $\varepsilon<1$ but for every partition one of the submatrices has the same $(1,q)$--norm as the whole matrix.
  
  \begin{theorem}
  For $n=2^{2k-1}$ there exists a $2k\times n$ -- matrix $A$ for which (\ref{a_i_j_cond}) holds for $\varepsilon^q\log_2{2n}\geq 2$, but for every partition of the form (\ref{partition}) the following equality holds:
  \begin{gather*}
  \max\biggl\{\left\|A(\Omega_1)\right\|_{(1, q)}, \left\|A(\Omega_2)\right\|_{(1, q)} \biggr\} = \left\|A\right\|_{(1, q)}.
  \end{gather*}
  \end{theorem}
\begin{proof}[Sketch of proof]
For every pair of the subsets $\omega$ and $\langle 2k\rangle\backslash \omega$ of the set $\langle 2k\rangle$ we choose the subset (any) of the largest cardinality. Let us numerate such subsets: $B_1, \ldots, B_{2^{2k-1}}$. 
We construct a matrix $A$ in the following way: if  $i\in B_j$, then 
$a_j^i=\frac{1}{|B_j|^{1/q}}$, otherwise, $a_j^i=0$. It is easy to check that (\ref{a_i_j_cond}) holds for $A$ and that for any partition (\ref{partition}) either $\left\|A(\Omega_1)\right\|_{(1,q)}=\left\|A\right\|_{(1,q)}$  or $\left\|A(\Omega_2)\right\|_{(1,q)}=\left\|A\right\|_{(1,q)}$.
\end{proof}

Let $q=\infty$ and $A$ be an arbitrary matrix. There is no partition that decreases (even a little)  $(X,\infty)$--norms of two submatrices. It is because one can find a row $v_{\sup}$  of the matrix $A$ such that
$
 \left\| A\right\|_{(X,\infty)}=\sup\limits_{\left\| x \right\|_{X}\leq 1}{\langle x, v_{\sup}\rangle}, 
$
and then the norm of the submatrix containing a row $v_{\sup}$, will be equal to the norm of $A$.

\bigskip
\thanks{The work was supported by the Russian Federation Government Grant No. 14.W03.31.0031.}

The paper is submitted to Mathematical Notes.

\end{document}